\documentclass[reqno]{amsart}

\usepackage[sorting=none, sorting=nyt, maxbibnames=99, backend=biber]{biblatex}

\renewbibmacro{in:}{}
\bibliography{references.bib}
\usepackage{amssymb}
\usepackage{calrsfs}
\usepackage{graphics,graphicx,mathrsfs}
\usepackage{enumerate}
\usepackage{enumitem}
\usepackage{textcomp}
\usepackage{bbold}
\usepackage{url}
\usepackage{xcolor}
\definecolor{vio}{rgb}{0.54, 0.17, 0.89}
\usepackage[ruled, lined, linesnumbered, longend]{algorithm2e}
\usepackage{hyperref}
\usepackage{titlesec}

\newtheorem{theorem}{Theorem}[section]
\newtheorem{lemma}[theorem]{Lemma}
\newtheorem{question}[theorem]{Question}
\newtheorem{proposition}[theorem]{Proposition}
\newtheorem{conjecture}[theorem]{Conjecture}

\numberwithin{equation}{section}

\theoremstyle{remark}

\titleformat{\section}
  {\normalfont\large\bfseries\centering}{\thesection}{1em}{}

\titleformat{\subsection}
  {\normalfont\bfseries}{\thesubsection}{1em}{}

\def\reals{\hbox{\rm I\kern-.18em R}}
\def\complexes{\hbox{\rm C\kern-.43em
\vrule depth 0ex height 1.4ex width .05em\kern.41em}}
\def\field{\hbox{\rm I\kern-.18em F}} %symbol for field

\let\svthefootnote\thefootnote
\newcommand\freefootnote[1]{%
  \let\thefootnote\relax%
  \footnotetext{#1}%
  \let\thefootnote\svthefootnote%
}

\newenvironment{section*}[2][A]{
  \section*{#2}
  \renewcommand\thesection{#1}
  \setcounter{theorem}{0}}{}

\newcommand{\rev}[1]{\overleftarrow{#1}}

\allowdisplaybreaks

\begin{document}

%\begin{center}
    %\huge{\textsc{Supporting data statement}}
%\end{center}
%\normalsize
%\
%\\
%\\
%\noindent
%To whom it may concern,\\
%\\
%The attached manuscript ``Power-free palindromes and reversed primes" has no associated data. There are however, some computational results which can readily be recovered using the code linked in \cite{djcode}.\\
%\\
%Regards,\\
%Shashi Chourasiya and Daniel R. Johnston
%\newpage 

\title[Power-free palindromes and reversed primes]{Power-free palindromes and reversed primes}

\author{Shashi Chourasiya and Daniel R. Johnston}
\address{School of Science, UNSW Canberra, Australia}
\email{s.chourasiya@unsw.edu.au}
\address{School of Science, UNSW Canberra, Australia}
\email{daniel.johnston@unsw.edu.au}
\date\today

\begin{abstract}
    We prove new results related to the digital reverse $\rev{n}$ of a positive integer $n$ in a fixed base $b$. First we show that for $b\geq 26000$, there exists infinitely many primes $p$ such that $\rev{p}$ is square-free. Further, we show that for $b\geq 2$ there are infinitely many palindromes (with $n=\rev{n}$) that are cube-free. We also give asymptotic expressions for the counting functions corresponding to these results. The main tools we use are recent bounds from the literature on reversed primes and palindromes in arithmetic progressions.
\end{abstract}

\maketitle

\freefootnote{\textit{Corresponding author}: Daniel Johnston (daniel.johnston@unsw.edu.au).}
\freefootnote{\textit{Affiliation}: School of Science, The University of New South Wales Canberra, Australia.}
\freefootnote{\textit{Key phrases}: digital reverse, palindromes, reversible primes, power-free}
\freefootnote{\textit{2020 Mathematics Subject Classification}: 11A63 (Primary) 11N37, 11Y35 (Secondary)}

\section{Introduction}
\subsection{Outline of paper}
In this paper, we are interested in results concerning the digital reverse of a number. To make this notion precise, we fix a base $b\geq 2$ and set
\begin{equation}\label{bNdef}
    \mathcal{B}_N=\{b^{N-1}\leq n< b^{N}: b\nmid n\}
\end{equation}
to be the set of $N$-digit numbers which are not divisible by $b$. The \emph{digital reverse} of an integer $n\in\mathcal{B}_N$ with base-$b$ expansion
\begin{equation*}
    n=\sum_{0\leq i<N}n_ib^{i}
\end{equation*}
is then defined by
\begin{equation}\label{revdef}
    \rev{n}=\sum_{0\leq i<N}n_i b^{N-1-i}.
\end{equation}
So for example, in base 10,
\begin{equation*}
    \rev{1234}=4321\quad\text{and}\quad\rev{878787}=787878.
\end{equation*}
We note that the condition that $b\nmid n$ in the definition \eqref{bNdef} of $\mathcal{B}_N$ is so that the last digit $n_0$ of $n$ is non-zero, and the operator $\rev{\cdot}$ is an involution on $\mathcal{B}_N$.

Throughout we also work with the more restrictive set 
\begin{equation*}
    \mathcal{B}_N^*=\{b^{N-1}\leq n<b^N:(n,b^3-b)=1\}.
\end{equation*}
By only considering $\rev{n}$ coprime to $b^3-b=b(b^2-1)$ we avoid several arithmetical relations between $n$ and $\rev{n}$ which would otherwise complicate our results and proofs. Here, the modulus $b$ is important since for any $m\leq N$, the residue of $\rev{n}\pmod{b^m}$ is determined by the first $m$ digits of $n$. Then, in terms of the modulus $b^2-1$,
\begin{equation*}
    \rev{n}\equiv b^{N-1}n\pmod{b^2-1}
\end{equation*}
since $b\equiv b^{-1}\pmod{b^2-1}$. In particular,
\begin{equation}\label{b2m1iff}
    (n,b^2-1)>1\ \text{if and only if}\ (\rev{n},b^2-1)>1.
\end{equation}

Historically, studies of the digital properties of numbers have been confined to the realm of elementary number theory. However, in recent years there has been a flurry of results obtained by applying deep analytical methods. Of particular note is Maynard's work \cite{maynard2019primes} in 2019, which showed that there are infinitely many primes missing a fixed digit in base 10. Other significant works include \cite{mauduit2010probleme} and \cite{swaenepoel2020prime}.

With regard to the digital reverse of numbers, we have the following two long-standing conjectures.

\begin{conjecture}\label{revcon}
    For every base $b\geq 2$, there are infinitely many reversible primes\footnote{In some literature, a reversible prime is called an \emph{emirp}. That is, ``prime" spelt backwards.}. That is, prime numbers $p$ such that $\rev{p}$ is also prime.
\end{conjecture}
\begin{conjecture}\label{palincon}
    For every base $b\geq 2$, there are infinitely many palindromic primes. That is, prime numbers $p$ such that $p=\rev{p}$.
\end{conjecture} 
Currently, both of these conjectures appear out of reach, particularly Conjecture \ref{palincon} which would imply Conjecture \ref{revcon}. From here onwards we focus on the following weakenings of Conjecture \ref{revcon} and Conjecture \ref{palincon}, which are more approachable with current methods.

\begin{conjecture}\label{revcon2}
    For every base $b\geq 2$, there are infinitely many primes $p$ such that $\rev{p}$ is square-free.
\end{conjecture}
\begin{conjecture}\label{palincon2}
    For every base $b\geq 2$, there are infinitely many square-free palindromes. That is, integers $n>0$ such that $n$ is square-free and $n=\rev{n}$. 
\end{conjecture}

Conjectures \ref{revcon2} and \ref{palincon2} simply weaken the primality conditions in Conjectures \ref{revcon} and \ref{palincon} to the property of being square-free. We note that Conjecture \ref{palincon2} is deceptively difficult due to the sparseness of palindromes, and has been mentioned several times in the literature (see e.g.\ \cite[p.\ 10]{banks2005prime} and \cite[p.\ 7]{dartyge2024reversible}).

Currently the best result towards Conjecture \ref{revcon2} appears in the recent work of Dartyge et\ al. \cite{dartyge2024reversible}. Here, the authors show that there are infinitely many numbers $n$ such that both $n$ and $\rev{n}$ are square-free \cite[Theorem 1.4]{dartyge2024reversible}. Their result is only proven in base $b=2$, although their techniques appear to routinely generalise to other bases. In this paper we qualitatively improve upon this result for large $b$. In particular, we are able to prove Conjecture \ref{revcon2} provided $b\geq 26000$.

\begin{theorem}\label{revthm}
    Let $b\geq 26000$. Then, there are infinitely many primes $p$ such that $\rev{p}$ is square-free. More generally, for $k\geq 2$, if 
    \begin{equation*}
        r_{b,k}(N)=\#\{\rev{p}\in\mathcal{B}^*_N:\rev{p}\ k\text{th}\ \text{power-free}\},
    \end{equation*}
    then there exists an effectively computable constant $c_1(b)>0$ such that
    \begin{equation}\label{rbeq}
        r_{b,k}(N)=\frac{1}{\zeta(k)}\prod_{p\mid b^3-b}\left(1-\frac{1}{p^k}\right)^{-1}\frac{\varphi(b)}{b}\frac{b^N}{\log b^N}\left(1+O_{b}\left(\frac{1}{N}\right)\right),
    \end{equation}
    where $\zeta$ is the Riemann zeta-function and $\varphi$ is the Euler totient function.
\end{theorem}
The asymptotic \eqref{rbeq} indicates that, after accounting for the standard arithmetical relations modulo $b$ and $b^2-1$, the condition of $\rev{p}$ being $k$th power-free is independent from $p$ being prime. To see this, we first note that ${1/\zeta(k)}$ is the natural density of $k$th power-free numbers, and the product over $p\mid b^3-b$ arises by removing primes $p\mid b$ and $p\mid b^2-1$ from the Euler product of $1/\zeta(k)$. Then, since we are only considering $\rev{p}\in\mathcal{B}_N^*$ with $(\rev{p},b)=1$, the leading digit of $p$ must be coprime to $b$. By the prime number theorem, the number of such $p\in\mathcal{B}_N$ is asymptotically given by $\frac{\varphi(b)}{b}\frac{b^N}{\log b^N}$.

The main tool required to prove Theorem \ref{revthm} is a variation of an estimate for reversed primes in arithmetic progressions due to Bhowmik and Suzuki \cite[Theorem 1.1]{bhowmik2024telhcirid}, which they refer to as \emph{Telhcirid's}\footnote{Telhcirid is Dirichlet spelt backwards.} theorem. In \cite{bhowmik2024telhcirid}, Telhcirid's theorem is only proven for bases $b\geq 31699$, but we have taken the opportunity to lower this to $b\geq 26000$ using a simple computational argument. 

Next, with regard to Conjecture \ref{palincon2}, currently the best result in this direction is due to recent work of Tuxanidy and Panario \cite{tuxanidy2024infinitude}. In \cite{tuxanidy2024infinitude}, it is proven that there are infinitely many palindromes with at most 6 prime factors. Trivially, this yields infinitely many palidromes that are 7th power-free. However, by combining Tuxanidy and Panario's equidistribution estimate with a Brun--Titchmarsh style result due to Banks and Shparlinski \cite{banks2005prime}, we are able to prove the following.

\begin{theorem}\label{palinthm}
    For all bases $b\geq 2$, there are infinitely many cube-free palindromes. More generally, for any $k\geq 3$, if
    \begin{equation*}
        \mathcal{P}_b^*(x)=\{n\leq x:(n,b^3-b)=1\ \text{and}\ n=\rev{n}\}
    \end{equation*}
    and
    \begin{equation*}
        p_{k,b}(x)=\#\{n\in\mathcal{P}_b^*(x):n\ \text{is}\ k\text{th}\ \text{power-free}\},
    \end{equation*}
    then there exists an effectively computable constant $c_2(b)>0$ such that,
    \begin{equation}\label{p4beq}
        p_{k.b}(x)=\frac{|\mathcal{P}_b^*(x)|}{\zeta(k)}\prod_{p\mid b^3-b}\left(1-\frac{1}{p^k}\right)^{-1}\left(1+O_{b}\left(\frac{1}{e^{c_2(b)\sqrt{\log x}}}\right)\right).
    \end{equation}
\end{theorem}
Therefore, whilst we can give a result for all bases $b\geq 2$ (unlike Theorem \ref{revthm}), more work is required to detect square-free palindromes and prove Conjecture \ref{palincon2}. As in Theorem \ref{revthm}, the asymptotic \eqref{p4beq} indicates that the condition  of $n\in\mathcal{P}_b^*(x)$ being a palindrome is independent of $n$ being $k$th power-free (for $k\geq 3$).

Notably, our proof of Theorem \ref{palinthm} is independent of the sieve-theoretic arguments used by Tuxanidy and Panario to obtain infinitely many palindromes with at most $6$ prime factors. Consequently, we are able to also prove the following.

\begin{theorem}\label{palinthm2}
    For all bases $b\geq 2$, there are infinitely many cube-free palindromes with at most $6$ prime factors.
\end{theorem}

We also remark that we have made different notational choices between Theorems \ref{revthm} and \ref{palinthm}. Most notably, Theorem \ref{revthm} is stated as an asymptotic over numbers of fixed digit length, whereas Theorem \ref{palinthm} is stated as a result over numbers less than $x$. Results in either of these forms are closely related, and our definitions and notation were chosen as to best agree with the existing literature on reversed primes and palindromes.

The structure of the paper is as follows. In Section \ref{revsect} we prove Theorem \ref{revthm}, and in Section \ref{palinsect} we prove Theorems \ref{palinthm} and \ref{palinthm2}. Then, in Section \ref{discsect} we provide further discussion. In particular, we discuss further possible improvements to our results, along with some related questions. An appendix is also included which gives the computational details for our refinement of Telhcirid's theorem.

\subsection{Comments on recent work}\label{recentworksect}
Since the time of writing, there has been influx of further work by other authors that complements and improves upon the results here. We make particular note of:
\begin{enumerate}
    \item In the preprint \cite{santana2025}, it is also shown that there are infinitely many cube-free palindromes (albeit for $b\geq 1100$) via a different Fourier-analytic approach.
    \item Via private correspondence, G. Bhowmik and Y. Suzuki informed us that they were writing up a further refinement of Telhcirid's theorem that holds for all bases $b\geq 2$. Their preprint was subsequently released \cite{bhowmik2025zsiflaw}. As a consequence of their result, one can extend the range of $b$ in Theorem \ref{revthm} to $b\geq 2$ and thereby prove Conjecture \ref{revcon2}.
    \item Dartyge, Rivat and Swaenepoel released a preprint \cite{dartyge2025prime} which, for all bases $b\geq 2$, establishes a computable constant $K_b$ such that there are infinitely many reversed primes with at most $K_b$ prime factors. This provides an answer to our Question \ref{revq} and establishes avenues for further work in improving their values of $K_b$.
\end{enumerate}

\section*{Acknowledgements}
We thank Igor Shparlinski, Yuta Suzuki and Timothy Trudgian for their helpful comments on an earlier version of this paper. Both authors' research was supported by an Australian Government Research Training Program (RTP) Scholarship.

\section{Proof of Theorem \ref{revthm}}\label{revsect}
In this section we prove Theorem \ref{revthm} on $k$th power-free reversed primes. To do so, we first define 
\begin{align}
    \rev{\pi}_N(a,d)&:=\sum_{\substack{\rev{p}\in\mathcal{B}_N\\\rev{p}\equiv a\ \text{(mod $d$)}}}1\label{revpidef}\\
    \rev{\pi}_N^*(a,d)&:=\sum_{\substack{\rev{p}\in\mathcal{B}_N^*\\\rev{p}\equiv a\ \text{(mod $d$)}}}1\label{revpistardef}.
\end{align}
In particular, $\rev{\pi}_N(a,d)$ counts the number of reversed primes $\rev{p}$ in an arithmetic progression $a\pmod{d}$, and $\rev{\pi}^*_N(a,d)$ adds the condition $(\rev{p},b^3-b)=1$. In \cite[Theorem 1.1]{bhowmik2024telhcirid} an asymptotic expression for $\rev{\pi}_N(a,q)$ is given, that is, ``Telhcirid's theorem". For our purposes, we require an asymptotic expression for $\pi_N^*(0,d)$ when $(d,b^3-b)=1$. To achieve this, we make some small changes to the argument in \cite{bhowmik2024telhcirid} and use the calculations in the appendix to extend the range of bases to $b\geq 26000$, as opposed to the range $b\geq 31699$ given in \cite{bhowmik2024telhcirid}.

\begin{lemma}\label{telhcirid2lem}
    Let $b\geq 26000$ and $d\geq 1$ with $(d,b^3-b)=1$. Then, there exists an effectively computable constant $c=c(b)\in(0,1)$ such that
    \begin{equation}\label{telhcirid2eq}
        \rev{\pi}_N^*(0,d)=\frac{1}{d}\frac{\varphi(b)}{b}\frac{b^N}{\log b^N}\left(1+O_b\left(\frac{1}{N}\right)\right)+O_b\left(\frac{b^N}{e^{c\sqrt{N}}}\right),
    \end{equation}
    provided
    \begin{equation}\label{drange}
        d\leq\exp(c\sqrt{N}).  
    \end{equation}
\end{lemma}
\begin{proof}
    In what follows we use the standard notation $e(x)=\exp(2\pi i x)$ when considering exponential sums. As per the approach in \cite[Section 8]{bhowmik2024telhcirid}, we write
    \begin{align}\label{thetabigeq}
        \overleftarrow{\pi}^{*}_{N}(0,d) & = \frac{1}{d} \sum_{0 \leq h < d} \sum_{\overleftarrow{p} \in \mathcal{B}^{*}_{N}} e \left(\frac{h \overleftarrow{p} }{d} \right)\notag\\
        & =\frac{1}{d} \sum_{\substack{0 \leq h < d \\ d \mid (b^{2}-1)b^{N}h}} \sum_{\substack{\overleftarrow{p} \in \mathcal{B}^{*}_{N}}} e \left(\frac{h \overleftarrow{p} }{d} \right)+\frac{1}{d} \sum_{ \substack{0 \leq h < d\\ d \nmid (b^{2}-1)b^{N}h}} \sum_{\substack{\overleftarrow{p} \in \mathcal{B}^{*}_{N}}} e \left(\frac{h\overleftarrow{p} }{d} \right).
    \end{align}
    We begin with the first term of \eqref{thetabigeq}, which will give us the main term of \eqref{telhcirid2eq}. Here, we note that since $(d,b^3-b)=1$,
    \begin{equation*}
        \frac{1}{d} \sum_{\substack{0 \leq h < d \\ d \mid (b^{2}-1)b^{N}h}} \sum_{\substack{\overleftarrow{p} \in \mathcal{B}^{*}_{N}}} e \left(\frac{h \overleftarrow{p} }{d} \right)=\frac{1}{d}\sum_{\rev{p}\in\mathcal{B}_N^*}1=\frac{1}{d}\sum_{\substack{p\in\mathcal{B}_N\\(\rev{p},b^3-b)=1}}1.
    \end{equation*}
    As discussed in the introduction, (see \eqref{b2m1iff}) $(p,b^2-1)>1$ if and only if ${(\rev{p},b^2-1)>1}$. Therefore, for a fixed $b\geq 2$, since there can only be finitely many primes $p$ with $(p,b^2-1)>1$, there are also finitely many $\rev{p}$ with $(\rev{p},b^2-1)=1$. In particular,
    \begin{equation*}
        \frac{1}{d}\sum_{\substack{p\in\mathcal{B}_N\\(\rev{p},b^3-b)=1}}1=\frac{1}{d}\sum_{\substack{p\in\mathcal{B}_N\\(\rev{p},b)=1\\ (\rev{p},b^2-1)=1}}1=\frac{1}{d}\sum_{\substack{p\in\mathcal{B}_N\\(\rev{p},b)=1}}1+O_b(1).
    \end{equation*}
    We now show that
    \begin{equation*}
        \sum_{\substack{p\in\mathcal{B}_N\\(\rev{p},b)=1}}1=\frac{\varphi(b)}{b}\frac{b^N}{\log b^N}\left(1+O_b\left(\frac{1}{N}\right)\right)+O_b\left(\frac{b^N}{e^{c'\sqrt{N}}}\right)
    \end{equation*}
    for some constant $c'\in(0,1)$. To see this, we split the sum over $p$ as follows:
    \begin{equation}\label{ispliteq}
        \sum_{\substack{p\in\mathcal{B}_N\\(\rev{p},b)=1}}1=\sum_{i=1}^{b-1}\sum_{\substack{ib^{N-1}\leq p<(i+1)b^{N-1}\\(\rev{p},b)=1}}1+O(1),
    \end{equation}
    where the $O(1)$ appears since we are removing the mild condition that $b\nmid p$ (i.e.\ $p\neq b$) in the definition of $\mathcal{B}_N$. Now, the condition $(\rev{p},b)=1$ is determined by the first digit of $p$. In the context of \eqref{ispliteq}, we have $(\rev{p},b)=1$ if and only if $(i,b)=1$. That is,
    \begin{align*}
        \sum_{i=1}^{b-1}\sum_{\substack{ib^{N-1}\leq p<(i+1)b^{N-1}\\(\rev{p},b)=1}}1&=\sum_{\substack{i=1\\(i,b)=1}}^{b-1}\sum_{ib^{N-1}\leq p<(i+1)b^{N-1}}1\\
        &=\sum_{\substack{i=1\\(i,b)=1}}^{b-1}\left[\pi((i+1)b^{N-1})-\pi(ib^{N-1})\right]+O(1),
    \end{align*}
    where $\pi$ is the usual prime counting function. By a sufficiently strong, effective version of the prime number theorem (e.g.\ \cite[Corollary 1.3]{johnston2023some}), there exists $c'\in(0,1)$ such that
    \begin{align*}
        \pi((i+1)b^{N-1})-\pi(ib^{N-1})&=\int_{ib^{N-1}}^{(i+1)b^{N-1}}\frac{1}{\log t}\mathrm{d}t+O\left(\frac{b^N N}{e^{c'\sqrt{N}}}\right)\\
        &=\frac{b^N}{\log b^N}\left(1+O_b\left(\frac{1}{N}\right)\right)+O\left(\frac{b^N N}{e^{c'\sqrt{N}}}\right)
    \end{align*}
    and thus
    \begin{equation*}
        \sum_{\substack{i=1\\(i,b)=1}}^{b-1}\left[\pi((i+1)b^{N-1})-\pi(ib^{N-1})\right]=\frac{\varphi(b)}{b}\frac{b^N}{\log b^N}\left(1+O_b\left(\frac{1}{N}\right)\right)+O_b\left(\frac{b^N N}{e^{c'\sqrt{N}}}\right),
    \end{equation*}
    as required. To complete the proof of the lemma, it suffices to show that the second term in \eqref{thetabigeq} satisfies
    \begin{equation}\label{secondtermreq}
        \frac{1}{d} \sum_{ \substack{0 \leq h < d\\ d \nmid (b^{2}-1)b^{N}h}} \sum_{\substack{\overleftarrow{p} \in \mathcal{B}^{*}_{N}}} e \left(\frac{h\overleftarrow{p} }{d} \right)\ll_b \frac{b^N N}{e^{c''\sqrt{N}}}
    \end{equation}
    for some $c''>0$ and $d\leq\exp(c''\sqrt{N})$. Provided $b$ is sufficiently large, \eqref{secondtermreq} follows by the same reasoning as that in \cite[\S 8--11]{bhowmik2024telhcirid}, again noting that there are only $O_b(1)$ primes $p$ such that $(\rev{p},b^2-1)=1$, and that the condition $(\rev{p},b)=1$ is completely determined by the first base-$b$ digit of $p$. As discussed in our appendix, $b\geq 26000$ suffices.
\end{proof}
We now prove Theorem \ref{revthm}.
\begin{proof}[Proof of Theorem \ref{revthm}]
    We write $\mu_k(n)$ to be the indicator function for $k$-free integers so that $\mu_k(n)=1$ if $n$ is $k$th power-free, and $\mu_k(n)=0$ otherwise. In terms of the standard M\"obius function $\mu$, we have
    \begin{equation*}
        \mu_k(n)=\sum_{d^k\mid n}\mu(d)
    \end{equation*}
    so that
    \begin{equation}\label{rbproofeq}
        r_{k,b}(N)=\sum_{\rev{p}\in\mathcal{B}_N^*}\mu_k(\rev{p})=\sum_{\rev{p}\in\mathcal{B}_N^*}\sum_{d^k\mid\rev{p}}\mu(d).
    \end{equation}
    We split the double sum in \eqref{rbproofeq} into $S_1(N)+S_2(N)$, with
    \begin{align}
        S_1(N)&=\sum_{d\leq N^2}\sum_{\substack{\rev{p}\in\mathcal{B}_N^*\\ d^k\mid\rev{p}}}\mu(d)\label{s1neq}\\
        S_2(N)&=\sum_{d>N^2}\sum_{\substack{\rev{p}\in\mathcal{B}_N^*\\d^k\mid\rev{p}}}\mu(d).\notag
    \end{align}
    To begin with, we bound $S_2(N)$ as
    \begin{equation*}
        |S_2(N)|\leq\sum_{d>N^2}\sum_{\substack{n\in \mathcal{B}_N\\ d^k\mid n}}1\ll\sum_{d>N}\frac{|\mathcal{B}_N|}{d^k}\ll\frac{|\mathcal{B}_N|}{N^{2(k-1)}}=O_{b}\left(\frac{b^N}{N^2}\right)
    \end{equation*}
    noting that $|\mathcal{B}_N|<b^N$ and $k\geq 2$. Hence, $S_2(N)$ is sufficiently small so it suffices to show that $S_1(N)$ satisfies the asymptotic in \eqref{rbeq}. To estimate $S_1(N)$, we write
    \begin{equation*}
        S_1(N)=\sum_{\substack{d\leq N^2\\ (d,b^3-b)=1}}\sum_{\substack{\rev{p}\in\mathcal{B}_N^* \\d^k\mid\rev{p}}}\mu(d)=\sum_{\substack{d\leq N^2\\ (d,b^3-b)=1}}\mu(d)\overleftarrow{\pi}_N^*(0,d^k)
    \end{equation*}
    with $\rev{\pi}_N^*$ as defined in \eqref{revpistardef}. Hence, by Lemma \ref{telhcirid2lem},
    \begin{align}
        S_1(N)& =\frac{\varphi(b)}{b}\frac{b^N}{\log b^N}\left(1+O_b\left(\frac{1}{N}\right)\right)\sum_{\substack{d\leq N^2\\ (d,b^3-b)=1}}\frac{\mu(d)}{d^k}+O\left(N^2\frac{b^N}{e^{c\sqrt{N}}}\right)\label{s1neq2}
    \end{align}
    for all $b\geq 26000$. To deal with the sum in \eqref{s1neq2}, we write
    \begin{align}\label{revsplit}
        \sum_{\substack{d\leq N^2 \\ (d, b^3-b)=1 }}\frac{\mu(d)}{d^{k}}&= \sum_{(d, b^3-b)=1 }  \frac{\mu(d)}{d^{k}}-\sum_{\substack{d> N^2 \\ (d, b^3-b)=1 }} \frac{\mu(d)}{d^{k}}.
    \end{align}
    By converting to an Euler product,
    \begin{align}\label{revinf}
        \sum_{(d, b^3-b)=1 }  \frac{\mu(d)}{d^{2}} = \frac{1}{\zeta(2)}\prod_{p\mid b^3-b}\left(1-\frac{1}{p^2}\right)^{-1}.
    \end{align}
    Then,
    \begin{equation}\label{revtail}
        \left|\sum_{\substack{d> N^2 \\ (d, b^3-b)=1 }} \frac{\mu(d)}{d^{k}}\right|\leq\sum_{d> N^2} \frac{1}{d^{k}}=O\left(\frac{1}{N}\right).
    \end{equation}
    Substituting \eqref{revinf} and \eqref{revtail} into \eqref{revsplit}, we see that \eqref{s1neq2} reduces to the claimed asymptotic \eqref{rbeq}, thereby completing the proof.
\end{proof}

\section{Proof of Theorem \ref{palinthm}}\label{palinsect}
In this section we prove Theorem \ref{palinthm} by utilising an equidistribution estimate due to Tuxanidy and Panario \cite{tuxanidy2024infinitude} and a Brun--Titchmarsh style result due to Banks and Shparlinski \cite{banks2005prime}. To begin with, for any base $b\geq 2$ we let
\begin{equation*}
    \mathcal{P}_b(x)=\{n\leq x:b\nmid n\ \text{and}\ n=\rev{n}\}
\end{equation*}
denote the set of base-$b$ palindromes less than or equal to $x$, and $\mathcal{P}_b^*(x)$ be as defined in the statement of Theorem \ref{palinthm}.

Since any $N$-digit palindrome is fully determined by its first $\lceil N/2\rceil$ digits, a simple combinatorial argument yields that
\begin{equation}\label{Pbasym}
    |\mathcal{P}_b(x)|\asymp_b\sqrt{x}.
\end{equation}
It turns out that the same is true for $|\mathcal{P}_b^*(x)|$, as proven in \cite{tuxanidy2024infinitude}.

\begin{lemma}[{\cite[Lemma 9.1]{tuxanidy2024infinitude}}]\label{pbstarlem}
    For any base $b\geq 2$, we have
    \begin{equation*}
        |\mathcal{P}_b^*(x)|\asymp_b\sqrt{x}
    \end{equation*}
\end{lemma}

Next we state Tuxanidy and Panario's equidistribution estimate.

\begin{lemma}[{\cite[Theorem 1.5]{tuxanidy2024infinitude}}]\label{equilem}
    For any $b\geq 2$ and $\varepsilon>0$, there exists $\sigma(b,\varepsilon)>0$ such that
    \begin{equation}\label{equieq1}
        \sum_{\substack{d\leq x^{1/5-\varepsilon}\\(d,b^3-b)=1}}\sup_{y\leq x}\max_{a\in\mathbb{Z}}\left|\sum_{n\in\mathcal{P}_b^*(y)}\left(\mathbb{1}_{n\equiv a \mathrm{(}\mathrm{mod}\ d\mathrm{)}}-\frac{1}{d}\right)\right|\ll_{b,\varepsilon}\frac{|\mathcal{P}_b^*(x)|}{e^{\sigma(b,\varepsilon)\sqrt{\log x}}}.
    \end{equation}
\end{lemma}

Here, although not explicitly stated in \cite{tuxanidy2024infinitude}, the constant $\sigma(b,\varepsilon)$ is effectively computable. Now, for our purposes, we will only need the following consequence of Lemma \ref{equilem}.

\begin{lemma}\label{equilem2}
    For any $b\geq 2$, $k\geq 1$ and $\varepsilon>0$, there exists $\sigma(b,\varepsilon)>0$ such that 
    \begin{align}\label{equieq2}
        \sum_{\substack{d\leq x^{(1/5-\varepsilon)/k}\\(d,b^3-b)=1}}\mu(d)\sum_{n\in\mathcal{P}_b^*(y)}\left(\mathbb{1}_{n\equiv a \mathrm{(}\mathrm{mod}\ d^k\mathrm{)}}-\frac{1}{d^k}\right)\ll_{b,\varepsilon}\frac{|\mathcal{P}_b^*(x)|}{e^{\sigma(b,\varepsilon)\sqrt{\log x}}}.
    \end{align}
\end{lemma}
\begin{proof}
    The left-hand side of \eqref{equieq2} is bounded by the left-hand side of \eqref{equieq1} (replacing $d$ with $d^k$) so that the desired bound follows.
\end{proof}
Finally, we need the following variant of a Brun--Titchmarsh style result due to Banks and Shparlinski \cite[Theorem 7]{banks2005prime}.
\begin{lemma}\label{brunlem}
    For all $b\geq 2$ and $d\in\mathbb{Z}_{>0}$, we have
    \begin{align}
        \sum_{\substack{n\in\mathcal{P}_b(x)\\d\mid n}}1&\ll_b \frac{|\mathcal{P}_b(x)|}{\sqrt{d}},\label{bruneq1}\\
        \sum_{\substack{n\in\mathcal{P}_b^*(x)\\d\mid n}}1&\ll_b \frac{|\mathcal{P}_b^*(x)|}{\sqrt{d}}\label{bruneq2}.
    \end{align}
\end{lemma}
\begin{proof}
    Let $P_b(N;d)$ denote the set of base-$b$ palindromes with $N$ digits that are divisible by $d$. Analagous to \eqref{Pbasym}, we have (see e.g.\ \cite[Lemma 3]{kobayashi2023counting})
    \begin{equation}\label{PbN1asym}
        P_b(N;1)\asymp_b b^{N/2}.
    \end{equation}
    Hence, in our notation, \cite[Theorem 7]{banks2005prime} is equivalent to
    \begin{equation*}
        |P_b(N;d)|\ll_b\frac{b^{N/2}}{\sqrt{d}}.
    \end{equation*}
    As a result,
    \begin{equation*}
        \sum_{\substack{n\in\mathcal{P}_b(x)\\d\mid n}}1\leq\sum_{1\leq N\leq\lceil\log_b(x)\rceil}P_b(N;d)\ll_b\frac{1}{\sqrt{d}}\sum_{1\leq N\leq\lceil\log_b(x)\rceil}b^{N/2}\ll_b\frac{\sqrt{x}}{\sqrt{d}}
    \end{equation*}
    so that \eqref{bruneq1} follows from \eqref{Pbasym}. The bound \eqref{bruneq2} follows identically using Lemma~\ref{pbstarlem}.
\end{proof}

Using the above lemmas, we now prove Theorem \ref{palinthm}.
\begin{proof}[Proof of Theorem \ref{palinthm}]
    We proceed in a similar fashion to the proof of Theorem \ref{revthm}. In particular, we begin by writing
    \begin{equation}\label{p4bfirstsum}
        p_{k,b}(x)=\sum_{n\in\mathcal{P}_b^*(x)}\sum_{d^k\mid n}\mu(d).
    \end{equation}
    Fix $\varepsilon<1/5$ (say $\varepsilon=1/10$). We split \eqref{p4bfirstsum} into $S_1(x)+S_2(x)$, with
    \begin{align}
        S_1(x)&=\sum_{\substack{d\leq x^{(1/5-\varepsilon)/k}\\(d,b^3-b)=1}}\sum_{\substack{n\in\mathcal{P}_b^*(x)\\d^k\mid n}}\mu(d),\label{s1def}\\
        S_2(x)&=\sum_{\substack{d>x^{(1/5-\varepsilon)/k}\\(d,b^3-b)=1}}\:\sum_{\substack{n\in\mathcal{P}_b^*(x)\\d^k\mid n}}\mu(d).\label{s2def}
    \end{align}
    Here, the condition $(d,b^3-b)=1$ has been vacuously added, noting that every $n\in\mathcal{P}_b^*(x)$ satisfies $(n,b^3-b)=1$. Now, Lemma \ref{equilem2} with $c_2(b):=\sigma(b,\varepsilon)$ gives
    \begin{align}
        S_1(x)&=\sum_{\substack{d\leq x^{(1/5-\varepsilon)/k}\\(d,b^3-b)=1}}\mu(d)\frac{|\mathcal{P}_b^*(x)|}{d^k}+O_b\left(\frac{|\mathcal{P}_b^*(x)|}{e^{c_2(b)\sqrt{\log x}}}\right)\notag\\
        &=|\mathcal{P}_b^*(x)|\left(\sum_{(d,b^3-b)=1}\frac{\mu(d)}{d^k}-\sum_{\substack{d>x^{(1/5-\varepsilon)/k}\\(d,b^3-b)=1}}\frac{\mu(d)}{d^k}\right)+O_{b}\left(\frac{|\mathcal{P}_b^*(x)|}{e^{c_2(b)\sqrt{\log x}}}\right).\label{s1expand}
    \end{align}
    By converting to an Euler product,
    \begin{equation}\label{s1sub1}
        \sum_{(d,b^3-b)=1}\frac{\mu(d)}{d^k}=\frac{1}{\zeta(k)}\prod_{p\mid b^3-b}\left(1-\frac{1}{p^k}\right)^{-1}.
    \end{equation}
    Then,
    \begin{equation}\label{s1sub2}
        \left|\sum_{\substack{d>x^{(1/5-\varepsilon)/k}\\(d,b^3-b)=1}}\frac{\mu(d)}{d^k}\right|\leq\sum_{d>x^{(1/5-\varepsilon)/k}}\frac{1}{d^k}\ll\frac{1}{x^{\frac{k-1}{k}(1/5-\varepsilon)}}\ll\frac{1}{x^{\frac{1}{2}(1/5-\varepsilon)}}.
    \end{equation}
    Substituting \eqref{s1sub1} and \eqref{s1sub2} into \eqref{s1expand} yields
    \begin{equation}\label{s1asympeq}
        S_1(x)=\frac{|\mathcal{P}_b^*(x)|}{\zeta(k)}\prod_{p\mid b^3-b}\left(1-\frac{1}{p^k}\right)^{-1}\left(1+O_{b}\left(\frac{1}{e^{c_2(b)\sqrt{\log x}}}\right)\right),
    \end{equation}
    which is the asymptotic expression we wish to prove for $p_{k,b}(x)$. Hence, to finish, it suffices to show that $S_2(x)$ can be absorbed into the error term of \eqref{p4beq}. This can be seen by applying \eqref{bruneq2} of Lemma \ref{brunlem}:
    \begin{align}\label{s2boundeq}
        |S_2(x)|\leq \sum_{\substack{d>x^{(1/5-\varepsilon)/k}\\(d,b^3-b)=1}}\:\sum_{\substack{n\in\mathcal{P}_b^*(x)\\d^k\mid n}}1\leq \sum_{\substack{d>x^{(1/5-\varepsilon)/k}\\(d,b^3-b)=1}}\frac{|\mathcal{P}_b^*(x)|}{d^{k/2}}\ll\frac{|\mathcal{P}_b^*(x)|}{x^{\frac{1}{6}(1/5-\varepsilon)}}.
    \end{align}
    where we have used that $k\geq 3$. The right-most bound in \eqref{s2boundeq} is smaller than the error term in \eqref{s1asympeq} so that the desired asymptotic for $p_{k,b}(x)$ follows.
\end{proof}

Now, to prove Theorem \ref{palinthm2}, we first give the precise statement of Tuxanidy and Panario's result on almost-prime palindromes. 
\begin{lemma}[{\cite[Theorem 1.4]{tuxanidy2024infinitude}}]\label{tuxthm}
    Let $b\geq 2$. Then
    \begin{equation*}
        \#\{n\in\mathcal{P}_b(x):\Omega(n)\leq 6,\: P^-(n)\geq x^{1/21}\}\asymp_b\frac{|\mathcal{P}_b(x)|}{\log x}.
    \end{equation*}
    where $\Omega(n)$ is the total number of prime factors of $n$, and $P^-(n)$ denotes the smallest prime factor of $n$.
\end{lemma}
\begin{proof}[Proof of Theorem \ref{palinthm2}]
    By Lemma \ref{tuxthm}, it suffices to prove that
    \begin{equation}\label{littleoeq}
        \#\{n\in\mathcal{P}_b(x):P^{-}(n)\geq x^{1/21},\: \exists d\in\mathbb{Z}_{>0}\ \text{with}\ d^3\mid n\}=o\left(\frac{|\mathcal{P}_b(x)|}{\log x}\right).
    \end{equation}
    This follows by an application of Lemma \ref{brunlem}. Namely, by \eqref{bruneq1}, the left-hand side of \eqref{littleoeq} is bounded above by
    \begin{equation*}
        \sum_{d\geq x^{1/21}}\sum_{\substack{n\in\mathcal{P}_b(x)\\d^3\mid n}}1\ll_b\sum_{d\geq x^{1/21}}\frac{|\mathcal{P}_b(x)|}{d^{3/2}}\ll\frac{|\mathcal{P}_b(x)|}{x^{1/42}}=o\left(\frac{|\mathcal{P}_b(x)|}{\log x}\right).
    \end{equation*}
    as required.
\end{proof}

\section{Further discussion}\label{discsect}
\subsection{Possible improvements}
In order to prove Conjecture \ref{revcon2}, one would have to increase the valid range of bases $b\geq 26000$ in Theorem \ref{revthm}. Certainly, one could expand on our computational argument in the appendix. However, we performed some rough calculations and found this method limits out around $b=25960$, so that no significant improvement is possible. Thus, to lower the base further, one would need to introduce new analytic techniques. For example, one could try to adapt some of the methods in \cite{maynard2019primes}, which were used to detect primes with restricted digits in base $b=10$ (also see the discussion in Section \ref{recentworksect}).

A proof of Conjecture \ref{palincon2} also appears within reach, as our proof of Theorem \ref{palinthm} falls just short of being able to detect square-free palindromes. In particular, if the bound in Lemma \ref{brunlem} could be improved to
\begin{equation*}
    \sum_{\substack{n\in\mathcal{P}_b^*(x)\\d\mid n}}1\ll_b \frac{|\mathcal{P}_b^*(x)|}{d^{1/2+\delta}}
\end{equation*}
for any $\delta>0$, then Conjecture \ref{palincon2} (and an analogous improvement to Theorem \ref{palinthm2}) would follow by the same argument as the proof of Theorem \ref{palinthm}. In addition, if one could improve upon Tuxanidy and Panario's method in \cite{tuxanidy2024infinitude} to prove the existence of infinitely many palindromes with at most 2 prime factors, say with
\begin{equation}\label{2primeeq}
    \#\{n\in\mathcal{P}_b(x):\Omega(n)\leq 2\}\gg_b\frac{|\mathcal{P}_b(x)|}{\log x},
\end{equation}
then Conjecture \ref{palincon2} would also readily follow. Here, we note that $O(|\mathcal{P}_b(x)|^{1/2})$ palindromes are perfect squares \cite[Theorem 1]{cilleruelo2009power}, which is of a much lower order than a bound of the form \eqref{2primeeq} or similar.

\subsection{Related problems}
In \cite{dartyge2024reversible}, Dartyge et al.\ prove that for infinitely many {$n\in\mathcal{B}_N$}, one has\footnote{In fact, Tuxanidy and Panario's result \cite{tuxanidy2024infinitude} gives an even stronger result, namely infinitely many palindromes $n$ with $\Omega(n)=\Omega(\rev{n})\leq 6$. However, we state Dartyge et. al's result here as it is more targeted to the problem at hand and gives a better estimate for the number of $n\in\mathcal{B}_N$ satisfying \eqref{omeganneq1}.}
\begin{equation}\label{omeganneq1}
    \max\{\Omega(n),\Omega(\rev{n})\}\leq 8
\end{equation}
and thus
\begin{equation}\label{omeganneq2}
    \Omega(n\rev{n})\leq 16
\end{equation}
in base $b=2$. In a similar vein to our Theorem \ref{revthm}, one could try to prove an analogous result to \eqref{omeganneq2}, fixing $n$ to be prime.

\begin{question}\label{revq}
    Can one prove the existence of an integer $K>0$ such that for all, or a range of bases $b\geq 2$,
    \begin{equation*}
        \Omega(\rev{p})\leq K
    \end{equation*}
    for infinitely many primes $p$?
\end{question}

An affirmative answer to Question \ref{revq} would be a deeper result (and closer to Conjecture \ref{revcon}) than Theorem \ref{revthm}. This is due to the fact that the set of integers with a bounded number of prime factors has a natural density of 0, whereas a proportion of $6/\pi^2\approx 0.608$ numbers are square-free. To prove Question \ref{revq} using classical sieve methods, one would need an equidistribution result for reversed primes in arithmetic progressions, as opposed to the pointwise bound given by Telhcirid's theorem \cite[Theorem 1.1]{bhowmik2024telhcirid}.

The methods in this paper could also be used to study the following Goldbach-like conjecture.

\begin{conjecture}[Hcabdlog's conjecture]\label{hcabcon}
    Consider a base $b\geq 2$. If $b$ is odd or $b=2$, then every sufficiently large even number $N$ can be expressed as
    \begin{equation}\label{hcabeq}
        N=\rev{p_1}+p_2
    \end{equation}
    where $p_1$ and $p_2$ are prime. Otherwise, if $b>2$ is even then every sufficiently large number $N$ (even or odd) can be expressed in the form \eqref{hcabeq}. 
\end{conjecture}
In base 10, a simple computation yields that $N=11$ is the only exception to \eqref{hcabeq} for $4\leq N\leq 10^6$. We also give further explanation for the condition that $N$ must be even if $b$ is odd or $b=2$. Firstly, if $b=2$ then the reverse of an odd prime $p$ must also be odd since the leading digit of $p$ is always $1$. Thus, provided $p_1,p_2\neq 2$, the sum $\rev{p_1}+p_2$ is always even. On the other hand, suppose that $b$ is odd. Then, the parity of a number $n$ in base $b$ is determined by the sum of its digits, and the sum of the digits of $n$ is always equal to the sum of the digits of $\rev{n}$. So, one again has that $\rev{p}$ is odd for all odd primes $p$. 

As an approximation to Conjecture \ref{hcabcon}, one could apply a variant of Telhcirid's theorem as in the proof of Theorem \ref{revthm} to count representations of large $N$ as
\begin{equation}\label{hcabsfeq}
    N=\rev{p}+\eta,
\end{equation}
where $p$ is prime and $\eta$ is square-free. In particular, if $h_b(N)$ represents the number of representations of $N$ in the form $\eqref{hcabsfeq}$, then
\begin{equation*}
    h_b(N)=\sum_{p\in\mathcal{B}_N}\mu^2(N-\rev{p})=\sum_{d<\sqrt{N}}\sum_{\substack{p\in\mathcal{B}_N\\ \rev{p}\equiv N\ \text{(mod}\ d^2\text{)}}}\mu(d),
\end{equation*}
which directly depends on estimates for $\rev{\pi}(N,d)$. Obtaining an asymptotic expression for $h_b(N)$ would be analogous to Estermann's classical result on the sum of a prime and a square-free number \cite{estermann1931representations}, which was proven in relation to the standard Goldbach conjecture.

\begin{section*}[A]{Appendix: Refining Telhcirid's theorem}
\setcounter{equation}{0}
    In this appendix, we prove that the stated bound $b\geq 26000$ is admissible in Telhcirid's theorem and our variant in Lemma \ref{telhcirid2lem}. To do so, we note that in \cite{bhowmik2024telhcirid}, it is proven that an asymptotic for $\rev{\pi}_N(a,d)$, defined in \eqref{revpidef}, holds provided
    \begin{equation}\label{alphabeq}
        \alpha_b:=\frac{\log(C_b)}{\log b}<\frac{1}{5},
    \end{equation}
    where $C_b>0$ is any number such that
    \begin{equation}\label{Cbeq}
        f(\theta):=\sum_{0\leq h<b}\min\left(b,\frac{1}{\left|\sin\pi\left(\frac{h}{b}+\theta\right)\right|}\right)\leq C_bb
    \end{equation}
    uniformly for all $\theta\in\mathbb{R}$. Namely, \eqref{alphabeq} implies a sufficiently strong bound for the error term for $\rev{\pi}_N(a,d)$ in \cite[Theorem 1.1]{bhowmik2024telhcirid}. In \cite[Lemma 6]{bhowmik2024telhcirid}, an analytic expression for $C_b$ is provided which gives that \eqref{alphabeq} holds for all $b\geq 31699$. However, by utilising the periodicity of $f(\theta)$, it is possible to perform a moderate computation to expand the range of $b$. The relevant code is given in the Github repository \cite{djcode}.
    \begin{proposition}
        The condition \eqref{alphabeq} holds for all $b\geq 26000$.
    \end{proposition}
    \begin{proof}
    The case $b\geq 31699$ is proven in \cite{bhowmik2024telhcirid}, so we restrict to $26000\leq b\leq 31698$. Let
    \begin{equation*}
        f_h(\theta)= \min\left(b,\frac{1}{\left|\sin\left(\pi\left(\frac{h}{b}+\theta\right)\right)\right|}\right)
    \end{equation*}
    so that
    \begin{equation*}
        f(\theta)=\sum_{0\leq h<b}f_h(\theta).
    \end{equation*}
    Combining \eqref{alphabeq} and \eqref{Cbeq}, we aim to show that for $26000\leq b<31698$ and all $\theta\in\mathbb{R}$, we have
    \begin{equation}\label{fthetabeq}
        f(\theta)<b^{6/5}.
    \end{equation}
    First we note that it suffices to consider $\theta\in[0,1/b]$ since if $\theta'=\theta+1/b$ then
    \begin{equation*}
        f(\theta')=\sum_{0\leq h<b}f_{h}(\theta')=\sum_{0\leq h<b}f_{h+1}(\theta)=f(\theta)
    \end{equation*}
    by the $1$-periodicity of $|\sin(\pi x)|$. We now divide the interval $[0,1/b]$ into $K\geq 2$ segments of length $1/Kb$:
    \begin{equation*}
        S_i=\left[\frac{i}{Kb},\frac{i+1}{Kb}\right],\quad i=0,1,\ldots,K-1.
    \end{equation*}
    For each value of $i$ and $h$, we then note that 
    \begin{equation}\label{fhmaxeq}
        \max_{\theta\in S_i}f_h(\theta)=\max\left\{f_h\left(\frac{i}{Kb}\right),f_h\left(\frac{i+1}{Kb}\right)\right\}.
    \end{equation}
    To see why \eqref{fhmaxeq} holds, we refer to Figure \ref{fig:DesmosLb}, which shows the general structure of the function $f_h(x)$. In particular, $f_h(x)$ is a uniform sequence of concave up arches, connected by straight line segments at height $b$ and of length
    \begin{equation*}
        L_b:=\frac{2}{\pi}\arcsin\left(\frac{1}{b}\right).
    \end{equation*}
    \begin{figure}[t]
        \centering
        \includegraphics[width=0.6\textwidth]{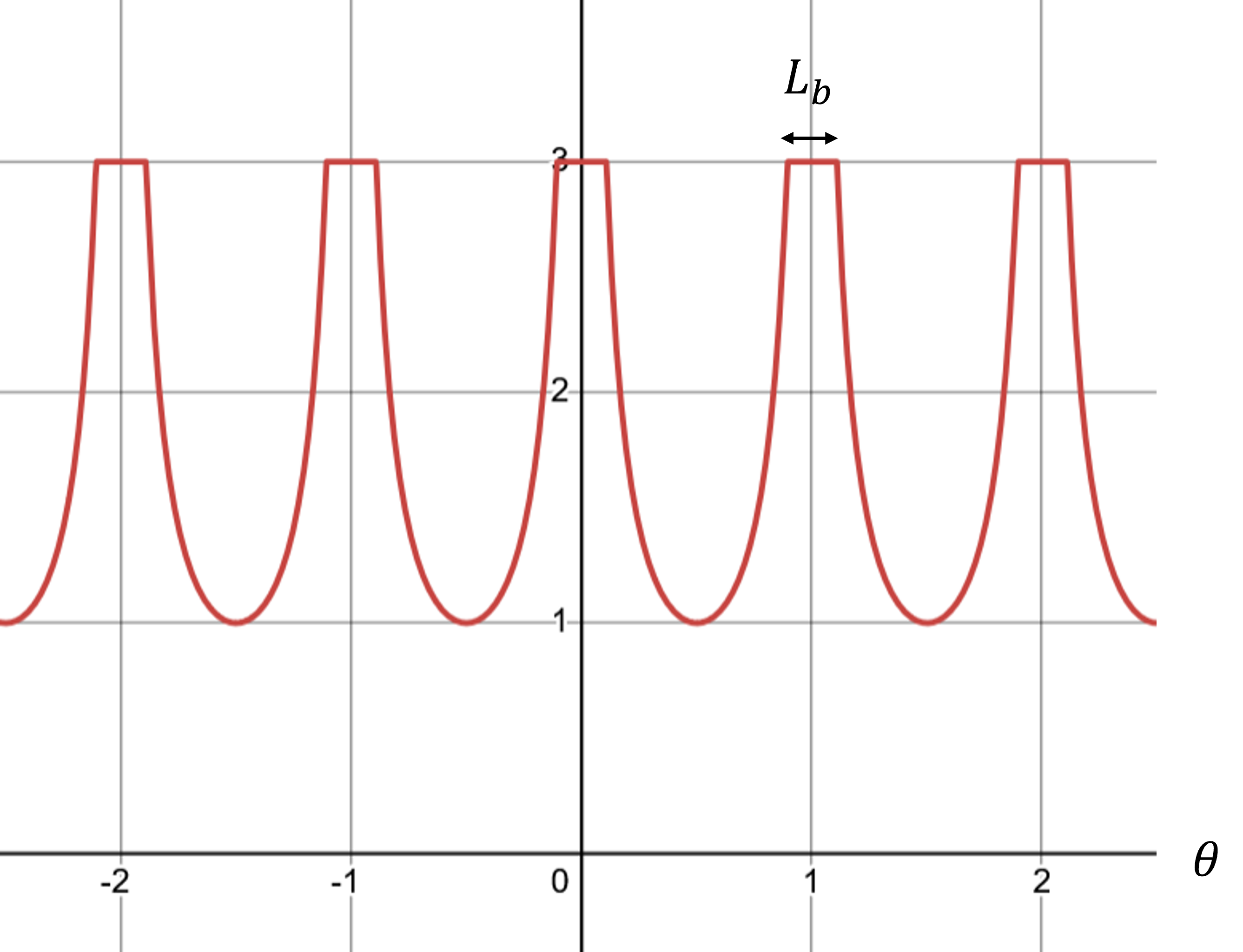}
        \caption{A plot of $f_h(\theta)$ generated by Desmos \cite{Desmos}. Here, $b=3$ and $h=0$. Increasing $b$ increases the height of the peaks and reduces $L_b$. Changing $h$ shifts the plot along the $\theta$-axis.}\label{fig:DesmosLb}
    \end{figure}
    Since
    \begin{equation*}
        \arcsin(x)=\int_0^x\frac{1}{\sqrt{1-t^2}}\mathrm{d}t>x
    \end{equation*}
    it follows that
    \begin{equation*}
        L_b>\frac{2}{\pi b},
    \end{equation*}
    which is greater than the length of $S_i$ for all $K\geq 2$. Thus, each interval $S_i$ is either:
    \begin{enumerate}[label=(\alph*)]
        \item Entirely within one of the concave up arches of $f_h(x)$,
        \item Entirely within one of the straight line segments of $f_h(x)$, or
        \item Partially within one of the arches, and partially within one of the adjacent straight line segments.
    \end{enumerate}
    In each case, the maximum of $f_h(\theta)$ occurs at an endpoint of $S_i$ as stated in \eqref{fhmaxeq}.

    Therefore, to verify \eqref{fthetabeq}, we compute an upper bound for $f(\theta)$ on each segment $S_i$ by using that
    \begin{equation}\label{fmaxeq}
        \max_{\theta\in S_i}f(\theta)\leq\sum_{0\leq h<b}\max_{\theta\in S_i}f_h(\theta),
    \end{equation}
    where the right-hand side of \eqref{fmaxeq} is evaluated using \eqref{fhmaxeq}. In Table \ref{table:bcomps} we provide, for different ranges of $b\in[b_0,b_1]$ a suitable value of $K$ such that the above computational procedure gives the inequality \eqref{fthetabeq}. This completes the proof.
\end{proof}

\def\arraystretch{1.2}
    \begin{table}[h]
    \centering
    \caption{Suitable values of $K$ to verify that \eqref{fthetabeq} holds for all bases $b$ with $b_0\leq b\leq b_1$. The time to check each range of $b$ is also included, as computed on a laptop with a 2.20 GHz processor.}
    \label{table:bcomps}
    \begin{tabular}{|c|c|c|c|} 
        \hline
        $b_0$ & $b_1$ & $K$ & Computation time (minutes) \\ [0.5ex] 
        \hline
        28500 & 31698 & 8 &  34 \\
        \hline
        26500 & 28499 & 34 & 55 \\
        \hline
        26100 & 26499 & 122 & 47 \\
        \hline
        26000 & 26099 & 367 & 35 \\
        \hline
    \end{tabular}
\end{table}

\newpage 

\printbibliography

@article{maynard2019primes,
  title={Primes with restricted digits},
  author={Maynard, J.},
  journal={Invent. Math.},
  volume={217},
  pages={127--218},
  year={2019},
}

@article{mauduit2010probleme,
  title={Sur un probl{\`e}me de {G}elfond: la somme des chiffres des nombres premiers},
  author={Mauduit, C. and Rivat, J.},
  journal={Ann. of Math.},
  pages={1591--1646},
  year={2010},
}

@article{swaenepoel2020prime,
  title={Prime numbers with a positive proportion of preassigned digits},
  author={Swaenepoel, C.},
  journal={Proc. London Math. Soc.},
  volume={121},
  number={1},
  pages={83--151},
  year={2020},
}

@article{dartyge2024reversible,
  title={Reversible primes},
  author={Dartyge, C. and Martin, B. and Rivat, J. and Shparlinski, I. E. and Swaenepoel, C.},
  journal={J. London Math. Soc.},
  volume={109},
  number={3},
  pages={e12883},
  year={2024},
}

@article{banks2005prime,
  title={Prime divisors of palindromes},
  author={Banks, W. D. and Shparlinski, I. E.},
  journal={Period. Math. Hungar.},
  volume={51},
  pages={1--10},
  year={2005},
  publisher={Springer}
}

@article{bhowmik2024telhcirid,
  title={On {T}elhcirid's theorem on arithmetic progressions},
  author={Bhowmik, G. and Suzuki, Y.},
  journal={preprint available at \href{https://arxiv.org/abs/2406.13334}{arXiv:2406.13334}},
  year={2024}
}

@article{tuxanidy2024infinitude,
  title={Infinitude of palindromic almost-prime numbers},
  author={Tuxanidy, A. and Panario, D.},
  journal={Int. Math. Res. Not.},
  volume={2024},
  number={18},
  pages={12466--12503},
  year={2024},
  publisher={Oxford University Press}
}

@misc{Desmos,
  title = {{Desmos Graphing Calculator}},
  howpublished = {available at \href{https://www.desmos.com/calculator}{www.desmos.com/calculator}}
}

@misc{djcode,
  author = {D. R. Johnston},
  year = {2025},
  howpublished = {Code for increasing the range of bases in Telhcirid's theorem, available at \href{https://github.com/DJmath1729/powerfreereverse}{github.com/DJmath1729/powerfreereverse}}
}

@article {estermann1931representations,
    AUTHOR = {Estermann, T.},
     TITLE = {On the representations of a number as the sum of a prime and a quadratfrei number},
   JOURNAL = {J. London Math. Soc.},
    VOLUME = {6},
      YEAR = {1931},
    NUMBER = {3},
     PAGES = {219--221},
}

@article{kobayashi2023counting,
  title={Counting relatively prime pairs of palindromes},
  author={Kobayashi, H. and Suzuki, Y. and Umezawa, R.},
  journal={preprint available at \href{https://arxiv.org/abs/2311.15002}{arXiv:2311.15002}},
  year={2023}
}

@article{cilleruelo2009power,
  title={Power values of palindromes},
  author={Cilleruelo, J. and Luca, F. and Shparlinski, I. E.},
  journal={J. Comb. Number Theory},
  volume={1},
  number={2},
  pages={101--107},
  year={2009},
}

@article{santana2025,
  title={Power-free integers and {F}ourier bounds},
  author={Carrillo Santana, S.},
  journal={preprint available at \href{https://arxiv.org/abs/2504.08502}{arXiv:2504.08502}},
  year={2025}
}

@article{bhowmik2025zsiflaw,
  title={The {Z}siflaw--{L}egeis theorem for arbitrary bases},
  author={Bhowmik, G. and Suzuki, Y.},
  journal={preprint available at \href{https://arxiv.org/abs/2507.08714}{arXiv:2507.08714}},
  year={2025}
}

@article{dartyge2025prime,
  title={Prime numbers with an almost prime reverse},
  author={Dartyge, C. and Rivat, J. and Swaenepoel, C.},
  journal={preprint available at \href{https://arxiv.org/abs/2506.21642}{arXiv:2506.21642}},
  year={2025}
}

@article{johnston2023some,
  title={Some explicit estimates for the error term in the prime number theorem},
  author={Johnston, D. R. and Yang, A.},
  journal={J. Math. Anal. Appl.},
  volume={527},
  number={2},
  year={2023},
}

\end{section*}

\end{document}